\documentclass[12pt]{amsart}

\usepackage{amsmath,mathrsfs}
\usepackage{amsthm}
\usepackage{amssymb}
\usepackage{enumerate}
\usepackage{xcolor}
\usepackage{comment}
\usepackage{tikz}
\usepackage{tikz-cd}
\usepackage{color} 
\definecolor{rossred}{rgb}{1.0,0.25,0.66}  
\definecolor{rossgreen}{rgb}{0.25,0.66,0.25} 
\definecolor{rossblue}{rgb}{0.25,0.66,1.0}
\definecolor{sashapurple}{rgb}{0.5,0.15,0.5}
\usepackage[pdfborder={0 0 0}]{hyperref}
\hypersetup{
    colorlinks=true,
    linkcolor=rossblue,
    citecolor=rossgreen
}

\numberwithin{equation}{section}
\theoremstyle{plain}
\newtheorem{theorem}{Theorem}[section]
\newtheorem{proposition}[theorem]{Proposition}
\newtheorem{lemma}[theorem]{Lemma}

\newtheorem{problem}[theorem]{Problem}
\newtheorem{corollary}[theorem]{Corollary}

\newtheorem{comments}[theorem]{Comments}

\newtheorem{notation}[theorem]{Notation}

\theoremstyle{definition}

\newtheorem{definition}[theorem]{Definition}

\newtheorem{remark}[theorem]{Remark}

\newtheorem{example}[theorem]{Example}
\AtBeginEnvironment{example}{%
  \pushQED{\qed}%
}
\AtEndEnvironment{example}{\popQED\endexample}

\usepackage{url}
\usepackage{enumitem}
\usepackage[utf8]{inputenc}
\usepackage[T1]{fontenc}
\usepackage[paperwidth = 8.5in, paperheight = 11in, inner = 1in, outer = 1in, top = 1in, bottom = 1in]{geometry}

\newcommand{\lex}{\text{lex}}

\newcommand{\ord}{\text{ord}}

\definecolor{campsc}{rgb}{0.44, 0.0, 1.0}

\title{Composite Linear Quotient Orderings of ideals and Modified Anticycles}

\author[]{Stephen Landsittel}
\address[S. Landsittel]
{Institute of Mathematics, Hebrew University, Givat Ram, Jerusalem 91904, Israel}
\email{stephen.landsittel@mail.huji.ac.il}

\date{\today}

\begin{document}

\begin{abstract}
In this paper we describe sufficient conditions for a power of a sum of two edge ideals to have linear quotients. We apply this result to give a class of modified anticycle graphs whose squared and cubed edge ideals have linear quotients.
\end{abstract}

\maketitle

\section{Introduction}

It is well known result of Herzog, Hibi, and Zheng, that a square-free quadric monomial ideal (i.e. an edge ideal) has a linear resolution if and only if all positive powers of that ideal have a linear resolution, this fact is proven in \cite{HHZ}. A classical problem which has been of great interest in algebra is the question of: When do sufficiently large powers of a monomial ideal $I$ have linear resolution? This question is also interesting in combinatorics, as a monomial ideal $I$ has a linear resolution if and only if the Alexander dual of its Stanley-Reisner complex is Cohen--Macaulay, which was proven in \cite{Eagon-Reiner}.

In studying this question regarding linear resolutions, the stronger property of $I$ having linear quotients was introduced in \cite{HT-lin_quot} and has been of recent interest. Villarreal and Van Tuyl discussed ideals with linear quotients in \cite{vantuyl_villrreal_shellable}. A monomial ideal $I$ has linear quotients if and only if the Alexander dual of its Stanley-Reisner complex is Shellable. Thus a question that has been of recent interest in the combinatorial algebra community is that of when a monomial ideal has linear quotients. Fr\"{o}berg proved in \cite{Froberg} that edge ideal has linear resolution if and only if it has linear quotients if and only if the corresponding graph of this edge ideal is cochordal. A well-known theorem of Dirac states that a finite graph is cochordal if and only if it is the $1$-skeleton of a quasi-tree. An algebraic proof of this fact is given in \cite{HHZ-dirac}. There is no known simple combinatorial description of when sufficiently large powers of an edge ideal have linear resolution, or even when the square of this ideal has a linear resolution, or linear quotients for that matter.

The property of a graph being gapfree has been observed to play a distinguished role in this story of linear resolution and linear quotients of edge ideals. It was observed that for a graph $G$, that if some power of its edge ideal $I_G$ has a linear resolution, then $G$ is a gapfree, see for instance \cite{Nevo-Peeva_C4}. In \cite{Nevo-Peeva_C4} they show that when $G$ is gapfree, $I_G^2$ need not have linear resolution. They conjecture that for a gapfree graph $G$, $I_G^s$ admits a linear resolution for all $s>>0$. More broadly speaking, over the recent decade, various versions of the following problem have been a substantial source of interest in algebraic combinatorics recently, which we now restate here, following a conjecture the Nevo and Peeva in \cite{Nevo-Peeva_C4}.

\begin{problem}\label{problem-main}
    Find a combinatorial characterization of the graphs $G$ for which (i) $I_G^s$ has a linear resolution for all $s>>0$; (ii) $I_G^s$ has linear quotients for all $s>>0$; (iii) $I_G^2$ has a linear resolution (or linear quotients).
\end{problem}

The general problem of finding conditions for when powers of a monomial ideal has a linear resolution, or linear quotients, has a vast history in the literature. In general it is difficult to determine whether a monomial ideal (even a small power of an edge ideal) has linear quotients. In Example 4.3 \cite{Conca2003} it is shown that the product of two monomial ideals with linear resolution need not itself have linear resolution. On the other hand, it is well known that a a graph is cochordal if and only if its corresponding edge ideal has linear quotients if and only if all powers of said ideal have linear quotients, by Fr\"{o}ger's theorem and Theorem 3.2 \cite{HHZ}.

Progress on Problem \ref{problem-main} has been made for a few classes of graphs. In \cite{BDMSM} they show that all powers at least two of edge ideals of anticycles have linear quotients. However, the matter of when powers above one, or even sufficiently large powers of edge ideals have linear quotients is largely mysterious. Some additional papers on linear quotient orderings of ideals are \cite{linquot-Jahan-Zheng} and \cite{linquot2}. With this result in mind we pose the following problem.
\begin{problem}\label{problem-2}
    Find a graph-theoretic operation $G\mapsto G'$ on connected graphs $G$ such that, if $I_G^s$ has linear quotients for $s>>0$, then $I_{G'}^s$ has linear quotients for $s>>0$.
\end{problem}
Given that Problem \ref{problem-main} is completely open, we present this Problem \ref{problem-2} with the hope in mind that over time the community can inductively expand the known families of graphs $G$ such that large enough powers of $I_G$ have linear quotients.

In this paper, we purpose two results in the direction of attacking this problem. Firstly we give a method for constructing linear quotient orderings of more complicated monomial ideals from simpler ones. More specifically we shall discuss linear quotient orderings of powers $s$ of edge ideals of graphs $H$, which are obtained using the following ingredients

(i) A subgraph $G_0$ of $H$ and a star $F_0\subset H$ which is coning over a vertex cover of $G_0$ such that $H = G_0\cup F_0$;

(ii) Linear quotient orderings of $I_{G_0}^{s-j}I_{F_0}^j$ for $j<s$.

(we are guaranteed to have an order in the $j=s$ case as a star is cochordal). We refer to this sort of construction as a \emph{composite linear quotient order} and we state the formal construction and result below.

\begin{proposition}[Construction of composite linear quotient orderings (Lemma \ref{lem-tech})]\label{propA}
    Let $G_0$ be a graph on $[n-1]$, let $F_0$ be a star graph on $[n]$, and let $s\in\mathbb{Z}_{\geq 2}$, and let $H_0$ be the graph on $[n]$ whose edge set is $E(G_0)\cup E(F_0)$. Assume the following two conditions
    \begin{enumerate}
        \item[(1)] $I_{G_0}^{s-j}I_{F_0}^{j}$ has linear quotients for $j<s$, by some ordering $\mathcal{O}_j$ of the minimal generators of $I_{G_0}^{s-j}I_{
        F_0}^j$. Let $\mathcal{O}_s$ be a linear quotient ordering of the minimal generators of $I_{F_0}^s$.
        \item[(2)] Each edge of $G_0$ is adjacent to some edge of $F_0$
    \end{enumerate}
    Then $I_{H_0}^s$ has linear quotients by the ordering
    \begin{equation*}
        \mathcal{O}
        :=
        (\mathcal{O}_s,\mathcal{O}_{s-1},\ldots,\mathcal{O}_{1},\mathcal{O}_0)
    \end{equation*}obtained by concatenating the orderings $\mathcal{O}_s,\ldots,\mathcal{O}_0$.
\end{proposition}

We give several examples of this machinery in action along with some applications to linear quotient orderings of some powers of certain edge ideals.

Secondly, we give a class of modifications of the anticycle graph on $n\geq 7$ vertices such that the square and cube of the corresponding edge ideal have linear quotients. Smaller orders than $7$ can be checked quickly by computer.

In particular, we give an explicit linear quotient ordering of the minimal generators of the edge ideal squared and cubed of the graph obtained from removing certain edges from the anticycle of order $n$. More precisely, we prove the following result, where $\overline{c}$ is the positive reside mod $n$ of an integer $c$.

\begin{theorem}\label{thmA}(Theorem \ref{cor1})
    Let $n\geq 7$ and let $a,b\in [n]:=\{1,\ldots,n\}$ such that $\lvert a-b\rvert$ is congruent to $2$ or $-2$ modulo $n$. Let $\mathcal{A}_n$ be the anticycle on $n$ vertices. Let $H$ be the graph obtained by removing the edges $\{a,b\}$ and $\{\overline{a+1},\overline{b+1}\}$ from $\mathcal{A}_n$ and adding the edge $\{\overline{a+1},a\}$. Then $I_G^2$ and $I_G^3$ have linear quotients, and hence admit a linear resolution.
\end{theorem}


Theorem \ref{thmA} follows from Lemma \ref{lem-main}, which we prove in Sections \ref{sec5} and \ref{sec6}.

Returning to Problem \ref{problem-2}, Theorem \ref{thmA} states that, for an anticycle $G$, then there is some family of modifications of $G$ consisting of adding and removing just a few edges which preserve the property that $I_G^2$ and $I_G^3$ have linear quotients. In general, there is a strict limit on how many edges one can remove from $G$ and preserve some power of $I_G$ having linear quotients, since no power can have linear quotients if $G$ has a gap. So if we remove too many edges, then we must add some edges back in the \emph{fill in} the gaps. Proposition \ref{propA} states if $F$ is a star obtained by coning over a vertex cover of $G$ and $I_G^{s-j}I_F^j$ has linear quotients for $j<s$, then $I_{G\cup F}^s$ has linear quotients.

The structure of this paper is as follows. In Section \ref{sec-prelim} we fix some notation and give some preliminaries in graph theory and commutative algebra. In Section \ref{sec:examples} we state some useful lemmas for producing linear quotient orderings and we give some examples that demonstrate some of the subtleties revolving around the theory of linear quotient orderings. In particular we give a computational method for showing when $I_G^2$ does not have linear quotients by the lexicographical order. In Section \ref{sec5} we construct some linear quotient orderings which will be used in the proof of Theorem \ref{thmA}, and in Section \ref{sec6} we finish the proof of Theorem \ref{thmA}, by showing that $I_{H_n}^2$ and $I_{H_n}^3$ have linear quotients by a patching of three and four lexicographical intervals of monomials, respectively.

\section{Preliminaries}\label{sec-prelim}

In this section we discuss some of the essential background needed from the theory of edge ideals and linear quotient orderings.

\subsection{Graph Theory}\label{sec2}

A \emph{graph} is a pair $G=(V,E)$ where $V$ is a set and $E$ is a set of doubletons of elements in $V$. Elements of $V$ are called \emph{vertices} and elements of $E$ are called \emph{edges}. We write $V(G):=V$ and $E(G):= E$. A graph is called \emph{edges} if $|V|<\infty$ and \emph{edges} if $u\neq v$ for $\{u,v\}\in E$. In this case, $|G|:=|V(G)|$ is called the \emph{order} of $G$. All graphs discussed in this paper will be finite and simple. A graph is \emph{connected} if for $u,v\in V$, there is a sequence of edges $\{u,v_1\},\ldots,\{v_{r-1},v_r\},\{v_r,v\}$. The \emph{degree} of a vertex $v\in V$ is the integer $\deg(v):= |\{i\in V\mid \{i,v\}\in E(G)\}|$.

For $n\geq 3$, the \emph{cycle} $C_n$ is the graph whose edge set is $\{\{i,i+1\}\mid i=1,\ldots,n-1\}\cup \{1,n\}$. A \emph{star} is a connected graph such that $\deg(v)=1$ for all but one $v\in V$. The \emph{compliment} $G^c$ of a graph $G$ is the graph $H$ such that $V(H) = V(G)$ and for $u,v\in V$ with $u\neq v$ we have $\{u,v\}\in E(H)$ if and only if $\{u,v\}\notin E(G)$. The compliment $\mathcal{A}_n$ of the cycle of order $n$ is called the \emph{anticycle} of order n. These two graphs coincide if $n=5$.

A \emph{subgraph} of a graph $G$ is a graph $G'$ for which $V(G')\subset V(G)$ and $E(G')\subset E(G)$, in this case we write $G' \subset G$. A subgraph $G'\subset G$ is \emph{induced} if $E(G') = \{\{i,j\}\in E(G)\mid i,j\in V(G')\}$. A graph $G$ is called \emph{chordal} if for any cycle $C_n\subset G$, $C_n$ is not induced unless $n = 3$. A graph $G$ is \emph{cochordal} if $G^c$ is chordal. For instance, any star is cochordal, this fact will be used in the proof of Theorem \ref{thmA}. For any two graphs $G_1$ and $G_2$ on vertex sets $V_1$ and $V_2$ respectively, which are both contained in some common vertex set $V$, we define the \emph{union} of $G_1$ and $G_2$ as the graph $H$ whose vertex set is $V(G_1)\cup V(G_2)\subset V$ and whose edge set is $E(G_1)\cup E(G_2)\subset \{\{a,b\}\mid a,b\in V,a\neq b\}$. Unions of graphs will play an important role throughout this paper.

\subsection{Commutative Algebra}\label{sec3}

Throughout this manuscript fix a field $K$ and let $S = K[x_1,\ldots,x_n]$ be the polynomial ring. The theory of linear quotient orderings is combinatorial, in the sense that we are generally not concerned about the nature of the ground field $K$. To a graph $G = (V,E(G))$, where $V = [n]:=\{1,\ldots,n\}$, we associate its \emph{edge ideal}
\begin{equation*}
    I_G:=(x_ix_j\mid \{i,j\}\in E(G))\subset S.
\end{equation*}

A monomial ideal $I\subset S$ is said to have \emph{linear quotients} if $I$ is principal, or for some ordering $M_1,\ldots,M_r$ of the minimal (monomial) generators of $I$, we have for $i\in \{2,\ldots,r\}$ that
\begin{equation*}
    Q_i:= (M_1,\ldots,M_{i-1}):M_i
\end{equation*} is generated by a subset of $\{x_1,\ldots,x_n\}$. In this case, we call $M_1,\ldots,M_r$ a linear quotient ordering of $I$.

It is well known that if a monomial ideal in $S$ has linear quotients, then it has a linear resolution. Now we recall the Herzog-Hibi-Zheng theorem.
\begin{theorem} (Theorem 3.2 \cite{HHZ})
    If a quadratic monomial ideal $I\subset S$ has linear quotients, then every power of $I$ has linear quotients.
\end{theorem}

In Proposition 4.1 \cite{BDMSM} they construct a particular linear quotient ordering for the edge ideal of whisker graphs.

In \cite{BDMSM} they also prove that, even though anticycles are not cochordal, high enough powers (in fact, at least two) of their edge ideals still have linear quotients.

\begin{theorem}(Theorem 5.6 \cite{BDMSM})
    For $m\geq 5$ and $s\geq 2$, the ideal $I_{\mathcal{A}_m}^s$ has linear quotients.
\end{theorem}

In \cite{BDMSM} they give a computer algorithm which can be of use for testing when powers of edge ideals $I_G^s$ have linear quotients for small $s$ when $G$ has a small number of edges. From the author's testing, beyond around $s=3$ and $|E(G)|=36$ this program is generally too computationally expensive for even current supercomputers, albeit it can be useful to test for which graphs $G$ we can expect $I_G^s$ to have linear quotients for at least small enough $s$. See Example \ref{ex2} for a calculation using the algorithm of \cite{BDMSM}.

\section{Notation and Examples of Linear Quotient Orderings}\label{sec:examples}
In this section we give some notation and lemmas which will be useful during the proof of Theorem \ref{thmA}, as well as some examples and non-examples of linear quotient orderings of graphs which are similar to the graphs of Theorem \ref{thmA}.

\subsection{General remarks and machinery to obtain linear quotient orderings}
We begin by introducing some notation which will be useful in proving Theorem \ref{thmA}. The following lemma is consistently useful in the construction of linear quotient orderings.


\begin{lemma}[Lemma 2.2 \cite{BDMSM}]\label{lem-LQ}
    Let $I=(M_1,\ldots,M_r)\subset S$ be a monomial ideal. Then the ordering $M_1,\ldots,M_r$ is a linear quotient ordering of $I$ if and only if:
    \begin{equation*}
        \begin{split}
            &\text{for $i,j\in\{1,\ldots,r\}$ with $j<i$ there exists $h<i$ (possibly equal to $j$) such that }\\
            &
            M_h/\gcd(M_h,M_i)\in\{x_1,\ldots,x_n\}\\
            &\text{ and }
            M_h/\gcd(M_h,M_i)\mid M_j/\gcd(M_j.M_i).
        \end{split}
    \end{equation*}
\end{lemma}

Sometimes in constructing a linear quotient ordering of a power of an edge ideal, it can be useful to use (some variant of) the lexicographical order of monomials for some parts of the ordering, which can sometimes then be glued together to form a total linear quotient order of the ideal, c.f. \cite{BDMSM}.

\begin{definition}\label{df1}
Let $<_{\lex}$ be the lexicographical ordering on the set of (monic) monomials in $S$ induced by the prescription that $x_n>x_{n-1}>\ldots>x_1$. That is 
    \begin{equation*}
        \begin{split}
            x_1^{\beta_1}\cdots x_n^{\beta_n}&<_{\lex}x_1^{\alpha_1}\cdots x_n^{\alpha_n} \text{ if and only if }
            (\beta_n,\beta_{n-1},\ldots,\beta_1)
            <_{\lex}
            (\alpha_n,\alpha_{n-1},\ldots,\alpha_1)
        \end{split}
    \end{equation*}
\end{definition}

For instance, the lexicographical order (from largest to smallest) of the monomials of degree three in $\mathbb{C}[x_1,x_2,x_3]$ is $x_3^3, x_3^2x_2,x_3^2x_1,x_3x_2^2,\ldots, x_1^3$.

The following remark is clear and we state it for convenience.
\begin{remark}\label{lem1}
    Let $M_1,M_2\in S$ be monomials such that $\deg(\gcd(M_1,M_2))= \deg(M_1) +\deg(M_2) - 1$.
    Then $M_3/\gcd(M_3,M_2)\in \{x_1,\ldots,x_r\}$.
\end{remark}


\subsection{The Anticycle Modification \texorpdfstring{$H_n$}{H\_n}}

In order to prove Theorem \ref{thmA} we shall consider the following modified version $H_n$ of the anticycle graph $\mathcal{A}_n$.

Fix an integer $n\geq 7$ throughout the rest of this paper. Let $G$ be the anticycle on $n-1$ vertices, which we will also denote by $\mathcal{A}_{n-1}$. Let $F$ be the star graph whose vertex set is $[n]$ and whose edge set is $\{\{i,n\}\mid 1\leq i\leq n-3\}$. Let $H_n$ be the graph obtained by removing the edges $\{n-2,n\}$ and $\{1,n-1\}$ from the anticycle of order $n$ and adding the edge $\{1,n\}$. We see that $E(H_n) = E(G) \cup E(F)$. $H_n$ will be important for the proof of \ref{thmA} since all graphs satisfying the hypothesis of \ref{thmA} are isomorphic to $H_n$.

\begin{remark}\label{rmk-gap1}
    The graph $H_5$ has the a gap, namely by the edges $\{1,3\}$ and $\{2,4\}$, and so $I_{H_5}^s$ does not have linear quotients for any $s\geq 2$. In fact it is clear from the definitions that $H_5\cong P_5$.
\end{remark}

It is of general interest, given a graph $H$ and a positive integer $s\geq 2$, whether $I_H^s$ has linear quotients by the lexicographical order of the variables (see Definition \ref{df1}). The graph $H_n$ is nontrivial in the following sense.

\begin{remark}\label{rmk-notlex1}
    If $n\geq 6$, then $I_{H_n}^2$ does not have linear quotients by the lexicographical order.
\end{remark}
\begin{proof}
Let $D_1 = x_1x_{n-3}x_{n-2}x_{n-1}$ and $D_3 = x_1^2x_{n-2}^2$. We may assume that $r\geq 6$ by Remark \ref{rmk-gap1}. Note that $D_1>_{\lex}D_3$ and the only linear divisors of
\begin{equation*}
    D_1/\gcd(D_1,D_3) = x_{n-3}^2x_{n-1}
\end{equation*}
are $x_{n-3}$ and $x_{n-1}$. If the lexicographic ordering is a linear quotient ordering, then by Lemma \ref{lem-LQ}, for some $i\in\{n-3,n-1\}$, and some variable $x\in \{x_1,\ldots,x_r\}$ dividing $D_3$ we would have that $M_3:=\frac{x_i}{x}D_3$ satisfies $M_3\in I_{G_r}^2$ and $M_3>_{\lex}D_3$. However
\begin{equation*}
    \begin{split}
        \frac{x_{n-3}}{x_1}D_3 &= x_1x_{n-2}^2x_{n-3}\notin I_{H_n}^2\text{, }
        \frac{x_{n-3}}{x_{n-2}}D_3 = x_1^2x_{n-3}x_{n-2}<_{\lex}D_3\text{, }\\
        \frac{x_{n-1}}{x_1}D_3 &= x_1x_{n-2}^2x_{n-1}\notin I_{H_n}^2\text{, and }
        \frac{x_{n-1}}{x_{n-2}}D_3 = x_1^2x_{n-2}x_{n-1}\notin I_{H_n}^2.
    \end{split}
\end{equation*}
This completes the proof of the remark.
\end{proof}

\subsection{Obstructions to Lexicographical Linear Quotients}

In this subsection we mention a computational method to prove that for suitable graphs $H$, to $I_H^2$ does not have linear quotients by the lexicographical order of its monomial generators. An important piece of this method involves linear quotient orders of the ideal $I_H^2$ which are \emph{piecewise} lexicographical (with a small number of pieces relative to the number of generators of the ideal), in a sense that we shall make precise below, see (\ref{eq-piecewise-LQ}).

Consider a family of graphs $\{H'_m\mid i\in \mathbb{Z}_{\geq 5}\}$ of order $m$ which are defined \emph{uniformly from anticycles}, i.e. by $q$-many additions or removals of edges from $\mathcal{A}_m$ where $q$ does not depend on $m$ (such as $\{H_m\mid m\geq 5\}$, where $q=3$). For such a family of graphs, we can use the \emph{Macaulay2 computer algebra software} and the findLinearOrderings method from \cite{BDMSM} to find a linear quotient order of $I_{H_m}^2$ for small $m$ (e.g. $m=5,6,7$) if $I_{H'_m}^2$ does not have linear quotients by the lexicographical order, then the program might output some ordering $\mathcal{O}$ which lexicographical except in a few places, where the number $p$ of places in our testing seems to be small if $q$ is small. More precisely that $\mathcal{O}$ has the form
\begin{equation}\label{eq-piecewise-LQ}
    \mathcal{O} = (M_{1,1}, M_{1,2},\ldots, M_{1,n_1}, M_{2,1},\ldots,M_{2,r_2},\ldots, M_{p+1,1},\ldots,M_{p+1,r_{p+1}})
\end{equation}
where $(M_{k,1},\ldots,M_{k,r_i})$ is lexicographically ordered for all $k$, but $M_{k,r_k}<_{\lex}M_{k+1,1}$ for each $k$.

Let $D_1>_{\lex}\ldots>_{\lex} D_{p'}$ be the \emph{obstruction monomials} $M_{1,r_1},M_{2,1},M_{2,r_2},\ldots,M_{p-1,r_{p-1}}$, $M_{p,1}$. Then one can attempt to check (if possible) that the condition of Lemma \ref{lem-LQ} is violated for all $h$, given some choice of indices $j_0<i_0$ and the monomials $M_j= D_{j_0}$ and $M_i:= D_{i_0}$ by checking all $h<i$ which have some chance to satisfy the hypothesis of Lemma \ref{lem-LQ} (i.e. similarly to the proof of Remark \ref{rmk-notlex1}). In this case, $I_{H'_i}^2$ cannot have linear quotients by the lexicographical order of its monomials by Lemma \ref{lem-LQ}. We demonstrate this method for a slightly different family of graphs in Remark \ref{rmk-notlex2}.

\begin{example}
    By computation in Macaulay2, we find that for $i=6$ and $7$, $H_i$ admits a linear quotient order of the above form $\mathcal{O}$ where
    \begin{equation*}
        \begin{split}
            \text{there are} &\text{ $3$ \emph{switches}, i.e. }p=3\text{, and }
            \text{the obstructions monomials are: }\\ &D_1 = x_1x_{n-3}x_{n-2}x_{n-1}, D_2 = x_1x_{n-3}^2x_{n-1},D_3 = x_1^2x_{n-2}^2\text{, and }D_4 = x_1^2x_{n-3}x_{n-2}.
        \end{split}
    \end{equation*}
\end{example}

\subsection{Further Anticycle Modifications and Their Orderings}

Continuing in the theme of linear quotient orderings of modifications of the anticycle graph (given by adding and removing a few edges), we conclude this section by giving some remarks regarding linear quotient orderings of a graph $G_n$, similar to the graph $H_n$ of Theorem \ref{thmA}. Let $G_n$ be the anticycle of order $n\geq 6$ with the edge $\{n-2,n\}$ removed.

We obtained the following calculation in Macaulay2.

\begin{example}\label{ex2}
    We see in the following Macaulay2 Calculation that $I_{G_6}^2$ has linear quotients thanks to the Macaulay2 method findLinearOrderings from \cite{BDMSM}.
    

\begin{enumerate}
    \item[i1]Q $=$ QQ$[x_1..x_6]$;
    \item[i2] I $= \text{ideal}(x_1*x_3,x_1*x_4,x_1*x_5,x_2*x_4,x_2*x_5,x_2*x_6,x_3*x_5,x_3*x_6)$;
    \item[i3]regularity I;
    \item[o4]$=3$ 
    \item[i5]regularity I$^2$;
    \item[o6]$=4$
    \item[i7]$\text{findLinearOrderings}(\text{I}^2,6)$
    
    Linear ordering found, returning as a list.
    
    -- Elapsed time: 2.45513 seconds.
    \item[o7]
    \begin{equation*}
        \begin{split}
            &= \{x_3^2x_6^2,x_2x_3x_6^2,x_2^2x_6^2,x_3^2x_5x_6,x_2x_3x_5x_6,x_1x_3x_5x_6,x_2^2x_5x_6,x_1x_2x_5x_6,x_2x_3x_4x_6,x_2^2x_4x_6,\\
            &x_1x_2x_4x_6,x_2x_2x_3x_6,x_2x_3^2x_6,x_1x_3x_4x_6,x_3^2x_5^2x_2x_3x_5^2,x_2^2x_5^2,x_1x_2x_5^2,x_1^2x_5^2,x_2x_3x_4x_5,x_1x_3x_4x_5,\\
            &x_2^2x_4x_5,x_1x_2x_4x_5,x_1^2x_4x_5,x_1x_3^2x_5,x_1x_2x_3x_5,x_1^2x_3x_5,x_2^2x_4^2,x_1x_2x_4^2,x_1x_2x_3x_4,x_1^2x_3x_4x_1^2x_4^2,x_1^2x_3^2\}
        \end{split}
    \end{equation*}
\end{enumerate}
\end{example}

\begin{remark}\label{rmk0}
     Calculation by Macaulay2 verifies that $I_{G_r}^3$ has linear quotients for $r=6,7,8$. Thanks to Adam Van Tuyl for running the $n=8$ case. On the other hand $G_5$ is not even gap free, and hence $I_{G_5}^j$ does not have linear quotients for any $j$. In fact the pair of edges $(\{1,3\}, \{2,5\})$ forms a gap in $G_5$.
\end{remark}

\begin{remark}\label{rmk-notlex2}
    For $r\geq 5$, the lexicographical ordering (induced by the monomial order $x_r>\ldots>x_1$) is not a linear quotient order of the minimal generators of $I_{G_r}^2$.
\end{remark}
\begin{proof}
Let $D_2:= x_1x_{r-3}^2x_r$ and $D_5:= (x_1x_{r-2})^2$. We may assume that $r\geq 6$ by Remark \ref{rmk0}. Note that $D_2>_{\lex}D_5$ and the only linear divisors of
\begin{equation*}
    D_2/\gcd(D_2,D_5) = x_{r-3}^2x_r
\end{equation*}
are $x_{r-3}$ and $x_r$. If our claim fails, then for some $i\in\{r-3,r\}$, and some variable $x\in \{x_1,\ldots,x_r\}$ dividing $D_5$ we would have that $M_3:=\frac{x_i}{x}D_5$ satisfies $M_3\in I_{G_r}^2$ and $M_3>_{\lex}D_5$. However
\begin{equation*}
    \begin{split}
        \frac{x_{r-3}}{x_1}D_5 &= x_1x_{r-3}x_{r-2}^2\notin I_{G_r}^2\text{, }
        \frac{x_{r-3}}{x_{r-2}}D_5 = x_1^2x_{r-3}x_{r-2}<_{\lex}D_5\text{, }\\
        \frac{x_r}{x_1}D_5 &= x_1x_{r-2}^2x_r\notin I_{G_r}^2\text{, and }
        \frac{x_r}{x_{r-2}}D_5 = x_1^2x_{r-2}x_r\notin I_{G_r}^2.
    \end{split}
\end{equation*}
\end{proof}

In the terminology of Problem \ref{problem-2}, the graph $G_n$ (up to graph isomorphism) is any modification of $\mathcal{A}_n$ obtained by removing an edge $\{a,b\}$ for which $|a-b|$ is congruent to $2$ modulo $n$.

\section{Proof of Theorem \ref{thmA}: Composite Linear Quotient Orderings}\label{sec5}

In this section we will introduce and prove several lemmas which will be useful in proving that Theorem \ref{thmA}. We accomplish this by decomposing the cube of the edge ideal of $H_n$ into a sum of four monomial ideals $I_{H_n}^3=I_1+I_2+I_3+I_4$. We then prove that $I_j$ has linear quotients for $1\leq j\leq 4$ and we carefully patch these four orderings into a \emph{composite} linear quotient order of $I_{H_n}^3$. The theoretical realization of this patching process is in Lemma \ref{lem-tech}.

\begin{notation}\label{not1}
    Let $\mathcal{O} = (\mathcal{E}, <)$ be a set of monomials $\mathcal{E}$ in $S$ with a fixed partial order $<$. To express that a monomial $M\in S$ is in the underlying set of $\mathcal{O}$ we will write $M\in \mathcal{O}$. Let $I\subset S$ be the ideal generated by the monomials of $\mathcal{O}$. Fix generators $M_1,M_2\in \mathcal{O}$ such that $M_1$ precedes $M_2$ in $\mathcal{O}$. To show that $\mathcal{O}$ is a linear quotient ordering of $I$, it suffices to show that there exits a generator $M_3\in \mathcal{O}$ such that
\begin{enumerate}
    \item[(i)] $M_3$ precedes $M_2$
    \item[(ii)] $M_3/\gcd(M_3,M_2)\in \{x_1,\ldots,x_n\}$, and
    \item[(iii)] $M_3/\gcd(M_3,M_2)\mid M_1/\gcd(M_1,M_2)$.
\end{enumerate}

We will say for conciseness that a given generator $M_3\in \mathcal{O}$ \textit{\emph{works in $\mathcal{O}$} with respect to $M_1$ and $M_2$} (or more concisely that \textit{$M_3$ \emph{works}}, if $M_1$, $M_2$, and $\mathcal{O}$ are understood) if conditions (i), (ii), and (iii) hold for that particular monomial $M_3$.
\end{notation}

Recall that $H_n$ is the graph obtained by removing the edges $\{n-2,n\}$ and $\{1,n-1\}$ from the anticycle of order $n$ and adding the edge $\{1,n\}$.

\begin{lemma}\label{lem-main}
    $I_{H_n}^s$ has linear quotients for $s=2,3$.
\end{lemma}

Before working on the proof of Lemma \ref{lem-main} we recall Theorem \ref{thmA} and demonstrate that Theorem \ref{thmA} follows from Lemma \ref{lem-main}.

\begin{theorem}\label{cor1}
    Let $\overline{c}$ be the positive reside mod $n$ of an integer $c$, and take $a,b\in [n]:=\{1,\ldots,n\}$ such that $\lvert a-b\rvert$ is congruent to $2$ or $-2$ modulo $n$. Let $\mathcal{A}_n$ be the anticycle on $n$ vertices. Let $H$ be the graph obtained by removing the edges $\{a,b\}$ and $\{\overline{a+1},\overline{b+1}\}$ from $\mathcal{A}_n$ and adding the edge $\{\overline{a+1},a\}$. Then $I_H^s$ has linear quotients for $s=2,3$. 
\end{theorem}
\begin{proof}
    By assumption, there is a graph isomorphism sending $a\mapsto n-2$ and $b\mapsto n$, so we may assume that $G = H_n$. Now the result follows from Lemma \ref{lem-main}
\end{proof}

The rest of this manuscript is dedicated to proving Lemma \ref{lem-main}. To prove Lemma \ref{lem-main} we begin by proving the following very general technical lemma, which allows one to produce linear quotient orderings of more complicated ideals from simpler ideals.
\begin{lemma}[Construction of composite linear quotient orderings]\label{lem-tech}
    Let $G_0$ be a graph on $[n-1]$, let $F_0$ be a star graph on $[n]$, and let $s\in\mathbb{Z}_{\geq 2}$, and let $H_0$ be the graph on $[n]$ whose edge set is $E(G_0)\cup E(F_0)$. Assume the following two conditions
    \begin{enumerate}
        \item[(1)] $I_{G_0}^{s-j}I_{F_0}^{j}$ has linear quotients for $0\leq j\leq s-1$, by some ordering $\mathcal{O}_j$ of the minimal monomial generators of $I_{G_0}^{s-j}I_{
        F_0}^j$. Let $\mathcal{O}_s$ be a linear quotient ordering of the minimal monomial generators of $I_{F_0}^s$.
        \item[(2)] Each edge of $G_0$ is adjacent to some edge of $F_0$
    \end{enumerate}
    Then $I_{H_0}^s$ has linear quotients by the ordering
    \begin{equation*}
        \mathcal{O}
        :=
        (\mathcal{O}_s,\mathcal{O}_{s-1},\ldots,\mathcal{O}_{1},\mathcal{O}_0)
    \end{equation*}obtained by concatenating the orderings $\mathcal{O}_s,\ldots,\mathcal{O}_0$.
\end{lemma}

\begin{remark}\label{rmk3}
    We see that property (2) in Lemma \ref{lem-tech} is satisfied when $G_0 = G$ and $F_0 = F$.
\end{remark}

\begin{proof}[proof of Lemma \ref{lem-tech}]
First we establish the claim that $I_{F_0}^s$ has linear quotients. By Fr\"{o}berg's theorem (see \cite{Froberg}) and the fact that $F_0$ is a star we see that $I_{F_0}$ has linear quotients. But then every power of $I_{F_0}$ has linear quotients by Theorem 3.2 \cite{HHZ}.

    Next we establish the claim that the set of minimal generators of $I_{H_0}^s$ is actually the union of the minimal generators of the ideals $I_{G_0}^{i}I_{F_0}^j$, $i,j\geq 0$, $i+j=s$ (this claim was stated explicitly at the end of the lemma). Since $E(H_0) = E(G_0)\cup E(F_0)$ we have that $I_{H_0} = I_{G_0}+I_{F_0}$. Then by the binomial theorem for ideals we have
    \begin{equation*}
        I_{H_0}^s = (I_{G_0}+I_{F_0})^s
        =\sum_{j=0}^sI_{G_0}^{s-j}I_{F_0}^j
    \end{equation*}and the claim follows. 


Let $\mathcal{M}$ be the set of minimal (monomial) generators of $I_{H_0}^s$ and take $M_1,M_2\in \mathcal{M}$ such that $M_1>_{\lex}M_2$ it suffices to find a monomial $M_3$ that works with respect to $M_1$ and $M_2$.
    
    By assumption, the result holds if $M_1,M_2\in \mathcal{O}_j$ for some $j$, so we may assume that $M_1\in \mathcal{O}_u$ and $M_2\in \mathcal{O}_v$ for some $s\geq u>v\geq 0$. Consequently, there exists
    \begin{equation*}
        \begin{split}
            &\{a_1,b_1\},\ldots,\{a_{s-v},b_{s-v}\}, \{i_1,j_1\},\ldots,\{i_{s-u},j_{s-u}\}\in E(G_0)\\
            \text{ and }&\{c_1,n\},\ldots,\{c_v,n\},\{r_1,n\},\ldots,\{r_u,n\}\in E(F_0)
        \end{split}
    \end{equation*}
    such that
    \begin{equation*}
        \begin{split}
            M_1 &= x_{c_1}\cdots x_{c_u}x_n^u
            x_{i_1}x_{j_1}\ldots x_{i_{s-u}}x_{j_{s-u}}\\
            M_2 &= x_{r_1}\cdots x_{r_v}x_n^v
            x_{a_1}x_{b_1}\ldots x_{a_{s-v}}x_{b_{s-v}}.\\
        \end{split}
    \end{equation*}Thus $x_n\mid M_1/\gcd(M_1,M_2)$. By assumption (2), one of $a_1$ or $b_1$ is adjacent to $n$, say $a_1$ is. Let
    \begin{equation*}
        M_3= \frac{x_{a_1}x_n}{x_{a_1}x_{b_1}}M_2.
    \end{equation*}We have that $M_3\in \mathcal{O}_{v+1}$ by construction, so that $M_3$ precedes $M_2$ in the ordering $\mathcal{O}$. On the other hand,
    \begin{equation*}
        M_3/\gcd(M_3,M_2) = x_n\mid M_1/\gcd(M_1,M_2).
    \end{equation*}
    Therefore $M_3$ works and our proof is complete.
\end{proof}

\begin{remark}\label{rmk2.1}
    For $s\geq 2$ we have that
    \begin{equation}
        I_{H_n}^s = \sum_{i+j=s,i,j\geq 0}I_{G}^iI_F^j.
    \end{equation}
\end{remark}
\begin{proof}
    By the definitions of $G$, $F$, and $H_n$ we have that
    \begin{equation*}
        E(H_n)=E(G)\cup E(F)
    \end{equation*}so that $I_{H_n}=I_G+I_F$. Now the result follows from the binomial theorem of ideals.
\end{proof}

From Remark \ref{rmk2.1} we see that
\begin{equation}\label{eq1}
    I_{H_n}^2= I_G^2+I_GI_F+I_F^2
\end{equation}

and
\begin{equation}\label{eq2}
    I_{H_n}^3= I_G^3+I_G^2I_F+I_GI_F^2+I_F^3
\end{equation}


\begin{remark}
     $I_F$ has linear quotients by the lexicographical order.
\end{remark}
\begin{proof}
    The set of minimal generators of $I_F$ is $\mathcal{F} = \{x_ix_n\mid 1\leq i\leq n-3\}$. Thus for $M_1,M_2\in \mathcal{F}$ we have that $M_1/\gcd(M_1,M_2)\in\{x_1,\ldots,x_n\}$.
\end{proof}

\begin{corollary}\label{cor3}
    $I_F^s$ has linear quotients for all $s\geq 1$.
\end{corollary}
\begin{proof}
    Since $F$ is a star, it is cochordal. Thus $I_F$ has linear quotients. Now the result follows since $I_F$ is a quadric monomial ideal.
\end{proof}

\begin{remark}
     $I_F^2$ has linear quotients by the lexicographical order.
\end{remark}
\begin{proof}
    The set of minimal generators of $I_F$ is $\mathcal{F} = \{x_ix_n\mid 1\leq i\leq n-3\}$. Let $M_1,M_2\in \mathcal{F}$ such that $M_1>_{\lex}M_2$. Write $M_1 = x_ix_nx_jx_n$ and $M_2 = x_ax_nx_bx_n$ with $j\geq i$ and $b\geq a$. Then either $j=b$ and $i>a$, so that $M_3:= x_ix_bx_n^2$ works, or $j>b$ and $M_3:=x_ax_jx_n^2$ works.
\end{proof}

\begin{lemma}\label{IGIF-lemma}
    $I_GI_F$ has linear quotients by the lexicographical order.
\end{lemma}
\begin{proof}
    Let $\mathcal{O}$ be the set of minimal (monic) monomial generators of $I_GI_F$ ordered by the lexicographical order. We have that
    \begin{equation*}
        \mathcal{O} = \{x_{i_1}x_{i_2}x_{i_3}x_n\mid
        1\leq i_3\leq n-3, \{i_1,i_2\}\in E(\mathcal{A}_{n-1})\}.
    \end{equation*}
    Fix $M_1,M_2\in \mathcal{O}$ such that $M_1>_{\lex}M_2$ and write
    \begin{equation*}
        \begin{split}
            M_1 &= x_ix_lx_jx_n\\
            M_2 &= x_ax_bx_cx_n
        \end{split}
    \end{equation*}
    where $\{i,l\},\{a,b\}\in E(\mathcal{A}_{n-1})$, $1\leq j,c\leq n-3$, $a<b$ and $i<l$. We must show that there exists $M_3\in\mathcal{O}$ which works with respect to $M_1$ and $M_2$. By Remark \ref{lem1} we may assume that $\deg(\gcd(M_1,M_2))\leq 2$, that is
    \begin{equation}\label{eq3}
        \#\{i,j,l\}\cap\{a,b,c,\}\leq 1.
    \end{equation}
    Additionally, since $M_1>_{\lex}M_2$ we have $(\dagger)$: $\max\{i,j,l\}\geq \max\{a,b,c\}$.

Now suppose that $j=\max\{i,j,l\}$ and $j\notin \{a,b,c\}$. We have $j>c$ by $(\dagger)$. Also $x_j\mid M_1/\gcd(M_1,M_2)$. Thus $M_3:= \frac{x_jx_n}{x_cx_n}M_2$ works.

Next suppose that $j=\max\{i,j,l\}$ and $j\in \{a,b,c\}$. By $(\dagger)$ and $a<b$, we have $j\in\{b,c\}$. Since $a<b$ we have $j\neq a$. Let $s = \max\{i,l\}$ and $q = \max\{a,c\}$. We have $s> q$ by the combined facts of (\ref{eq3}) and $M_1>_{\lex}M_2$. If $j=  b$, then $n-2\geq j\geq s>c$ so that $M_3:= \frac{x_sx_n}{x_cx_n}M_2$ works. Thus we can assume that $j = c$. then by $(\dagger)$ we have $l>b$. We have
\begin{equation*}
        \begin{split}
            M_1 &= x_ix_jx_lx_n\\
            M_2 &= x_ax_jx_bx_n
        \end{split}
    \end{equation*}
    so that $M_3:= \frac{x_lx_n}{x_bx_n}M_2$ works.

We have proven that the lemma holds when $j = \max\{i,j,l\}$ $(=\max\{j,l\})$. Thus we may assume that $l = \max\{i,j,l\}$. Note that $l\neq a$ by $(\dagger)$. Suppose that $l\notin\{b,c\}$. If $(a,l)\neq (1,n-1)$ then $M_3:=\frac{x_ax_l}{x_ax_b}M_2$ works, so we can assume that $(a,l)=(1,n-1)$. Then we have $i\neq 1$, since $\{i,n-1\}\in E(\mathcal{A}_{n-1})$. Now if $b\neq n-2$, then $M_3:=\frac{x_lx_b}{x_ax_b}M_2$ works, so we may further assume that $b=n-2$. If $i\neq n-3$ then $M_3:=\frac{x_ix_{n-2}}{x_1x_{n-2}}M-2$ works, so we can assume in addition that $i = n-3$. If $c\neq n-3$ then $M_3:= \frac{x_{n-3}x_n}{x_cx_n}M_2$ works, so we can assume further that $c=n-3$ Now we have
\begin{equation*}
        \begin{split}
            M_1 &= x_{n-3}x_{n-1}x_jx_n\\
            M_2 &= x_1x_{n-2}x_{n-3}x_n
            = x_{n-3}x_{n-2}\cdot x_1x_n
        \end{split}
    \end{equation*}
    so that $M_3:= \frac{x_{n-1}}{x_{n-2}}M_2 = x_{n-3}x_{n-1}x_1x_n$ works.

Now we have reduced to the case that $l = \max\{i,j,l\}$ and $l\in\{b,c\}$. Suppose that $l = b$. Then by (\ref{eq3}) we have that $\{i,j\}\cap \{a,c\} = \emptyset$. Since $M_1>_{\lex}M_2$ we also have that $\max\{i,j\}>\max\{a,c\}$. Thus either $i>a$ or $j>c$. If $i>a$, then $M_3:=\frac{x_ix_l}{x_ax_l}M_2$ works. If $j>a$, then $M_3:=\frac{x_jx_n}{x_cx_n}M_2$ works.

Now we can assume that $c=l = \max\{i,j,l\}$. Since $\max\{i,j\}>\max\{a,c\}$. Let $z =\max\{i,j\}\leq n-3$. Then, $z>a,b$ so that $\frac{x_ax_z}{x_ax_b}$ works. This completes the proof of the lemma.
\end{proof}

\begin{corollary}\label{cor2}
    $I_{H_n}^2$ has linear quotients.
\end{corollary}

\begin{proof}
    The result follows from Lemma \ref{lem-tech} by taking $G_0 = G$, $F_0 = F$, $H_0 = H_n$, equipping the minimal generators of $I_G^2$ with the linear quotient order of Theorem 5.4 \cite{BDMSM}, equipping the minimal generators of $I_GI_F$ with the lexicographical order, and equipping the minimal generators of $I_F^2$ with its lexicographical order.
\end{proof}

We will prove the next two lemmas in Section \ref{sec6}. These are the remaining two ingredients we need to prove Lemma \ref{lem-main}

\begin{lemma}\label{lem2}
    $I_GI_F^2$ has linear quotients by the lexicographical order.
\end{lemma}

\begin{lemma}\label{lem3}
    $I_G^2I_F$ has linear quotients by the lexicographical order.
\end{lemma}

We see that Lemma \ref{lem-main} follows from \ref{lem-tech}, Remark \ref{rmk3}, Corollary \ref{cor2}, Corollary \ref{cor3}, Lemma \ref{lem2}, and Lemma \ref{lem3}.

\section{Proof of Theorem \ref{thmA}: Demonstration of the Main Lemmas}\label{sec6}

In this section we finish the proof of Lemma \ref{lem-main}, and hence of Theorem \ref{cor1}, by proving Lemmas \ref{lem2} and \ref{lem3}. Before proving these lemmas we establish several facts about the lexicographical order of monomials from Notation \ref{df1}, which will be useful throughout the proof of Lemmas \ref{lem2} and \ref{lem3}.

We begin by proving Lemma \ref{lem1.8} which is a general machinery for obtaining linear quotient orderings of monomial ideals, followed by a few remarks. Lemma \ref{lem1.8} will be used a multitude of times implicitly throughout the proofs of Lemmas \ref{lem2} and \ref{lem3}. For a monomial $M\in S$ and $1\leq i\leq n$, let $\ord_{x_i}M$ be the largest natural number $m$ such that $x_i^m\mid M$.

\begin{lemma}[Lexicographical projections]\label{lem1.8} Let $M_1$ and $M_2$ be monic monomials in $S$ of the same degree $s\geq 3$ such that $M_1>_{\lex}M_2$. Write
\begin{equation*}
    \begin{split}
        M_1 &= x_{i_1}\cdots x_{i_s}\\
        M_2 &= x_{j_1}\cdots x_{j_s}
    \end{split}
\end{equation*}
where $i_1\leq \ldots \leq i_s$ and $j_1\leq \ldots \leq j_s$. For $1\leq t\leq s$ say that $M_1$ and $M_2$ agree in order $t$ if $i_l=j_l$ for $s-t+1\leq l\leq s$ and $i_{s-t}\neq j_{s-t}$. Say that $M_1$ and $M_2$ agree in order zero if $i_s>j_s$. Then we have the following statements.

\begin{enumerate}
    \item[(i)]There is a unique natural number $t= t(M_1,M_2)$ such that $M_1$ and $M_2$ agree in order $t$.
    \item[(ii)] Take $t = t(M_1,M_2)$ and let $j\in \{1,\ldots,s-t\}$ and let $M_3 = \frac{x_{i_{s-t}}}{x_j}M_2$. Then $M_1>_{\lex}M_3$, $M_3/\gcd(M_3,M_2)\in\{x_1,\ldots,x_n\}$, and $M_3/\gcd(M_3,M_2)\mid M_1/\gcd(M_1,M_2)$.
    \item[(iii)] Take $j$ and $M_3$ as in (ii). Let $\mathcal{O}$ be the lexicographical ordering on a given set $\mathcal{E}$ of monomials in $S$. In the terminology of Notation \ref{not1} we have that $M_3$ works in $\mathcal{O}$ with respect to $M_1$ and $M_2$ if and only if $M_3\in \mathcal{E}$.
\end{enumerate}
\end{lemma}
\begin{proof}
    (i) follows from the definition of the agreement in order $t$ and the assumption that $M_1>_{\lex}M_2$. (iii) follows immediately from (ii) and the definition in Notation \ref{not1} of working in $\mathcal{M}$ with respect to $M_1$ and $M_2$. Now we prove (ii).

    Since $i_1\leq \ldots \leq i_s$, $j_1\leq \ldots \leq j_s$, and $i_{s-t}>j_{s-t}$, we have that $i_{s-t}>j_1,\ldots,j_{s-t-1}$ so that $i_{s-t}\notin\{j_1,\ldots,j_{s-t}\}$. Thus $\ord_{x_{i_{s-t}}}M_1>\ord_{x_{i_{s-t}}}M_2$. Consequently
    \begin{equation*}
        x_{i_{s-t}} = M_3/\gcd(M_3,M_2)
    \end{equation*} and
    \begin{equation*}
        x_{i_{s-t}}\mid M_1/\gcd(M_1,M_2).
    \end{equation*} Now it remains to prove that $M_3>_{\lex}M_2$. We have by construction that
    \begin{equation*}
        M_3=x_{i_s}\cdots x_{i_{s-t}}x_{j_{s-t-1}}\cdots x_{j_1}\text{ and }
        M_2=x_{i_s}\cdots x_{i_{s-t+1}}x_{j_{s-t}}x_{j_{s-t-1}}\cdots x_{j_1}
    \end{equation*}and we are done since $i_1\leq \ldots \leq i_s$, $j_1\leq \ldots \leq j_s$, and $i_{s-t}>j_{s-t}$.
\end{proof}

We shall refer to the monomials $M_3$ constructed in Lemma \ref{lem1.8} as \textit{projections of $M_1$ onto $M_2$}. In the proof of Lemma \ref{lem3} we consider several cases depending on the order $t = t(M_1,M_2)$ of agreement. In most of these cases of Lemma \ref{lem3}, we show that some projection $M_3$ of $M_1$ onto $M_2$ works by way of Lemma \ref{lem1.8}(iii) and Remark \ref{rmk1.7}, which we state below.

Recall that we have fixed a positive integer $n\geq 6$, the anticycle of order $n-1$, denoted by $G:=\mathcal{A}_{n-1}$, and the star $F$ on $n$ vertices with edges $\{i,n\}$, $1\leq i\leq n-3$.

\begin{remark}\label{rmk1.7}
    If $\{i_1,i_2\}\in E(G)$ and $i_2\leq j_2<n-1$, then $\{i_1,j_2\}\in E(G)$.
\end{remark}
\begin{proof}
    The statement follows from the definition of the graph $G = \mathcal{A}_{n-1}$.
\end{proof}

The following remark is clear from the definition of the lex ordering of monomials.

\begin{remark}\label{rmk0.1}
    Consider the set $\mathcal{M}$ of (monic) monomials in $S$ equipped with the lexicographical order and let $M,M_1,M_2\in \mathcal{M}$. Then $M_1>_{\lex}M_2$ if and only if $MM_1>_{\lex}MM_2$.
\end{remark}

Now we prove Lemma \ref{lem2}.

\begin{proof}
    Let $\mathcal{O}$ be the set of minimal (monomial) generators of $I_GI_F^2$ equipped with the lexicographical order. For $M,N\in \mathcal{O}$ we shall write $N>M$ to express the statement that $N>_{\lex}M$. Now fix $M_1,M_2\in\mathcal{O}$ such that $M_1>M_2$. We must find $M_3$ which works with respect to $M_1$ and $M_2$. Write $M_1 = \alpha_1x_rx_n$ and $M_2 = \alpha_2x_dx_n$ with $\alpha_1,\alpha_2\in I_GI_F$ and $1\leq r,d\leq n-3$.

Suppose that $r = d$. By Lemma \ref{IGIF-lemma} there exists $\alpha_3\in I_GI_F$ and $t\in\{1,\ldots,n\}$ such that $\alpha_3>\alpha_2$, $x:= \alpha_3/\gcd(\alpha_3,\alpha_2) =x_t$, and
\begin{equation*}
    x_t
    \mid \alpha_1/\gcd(\alpha_1,\alpha_2).
\end{equation*}
Thus $\ord_{x_t}\alpha_1>\ord_{x_t}\alpha_2$, so that
\begin{equation*}
    \ord_{x_t}M_1=\ord_{x_t}\alpha_1x_rx_n>\ord_{x_t}\alpha_2x_rx_n=\ord_{x_t}M_2
\end{equation*}which implies that
\begin{equation*}
    x_t\mid M_1/\gcd(M_1,M_2).
\end{equation*}Thus $M_3:= x_dx_n\alpha_3$ works. So we can assume that $r\neq d$. By a similar argument we can assume that
\begin{equation*}
    \begin{split}
        M_1 &= x_ix_lx_jx_nx_rx_n\\
        M_2 &= x_ax_bx_cx_nx_dx_n 
    \end{split}
\end{equation*}where
\begin{equation}\label{eq4}
    \{j,r\}\cap\{c,d\}=\emptyset
\end{equation}$\{i,l\},\{a,b\}\in E(G)$, $i<l$, $a<b$, $1\leq j,r,c,d\leq n-3$, $j\leq r$, and $c\leq d$. By Lemma \ref{lem1} we can also assume that
\begin{equation}\label{eq5}
    \deg(\gcd(x_ix_lx_jx_r,x_ax_bx_cx_d))\leq 2.
\end{equation}

We introduce some notation before proceeding further with the proof. Let
\begin{equation*}
    \begin{split}
        u &= \max\{i,l,j,r\}\text{ } (= \max\{l,j,r\})\\
        v &= \max\{a,b,c,d\}\text{ } (= \max\{b,c,d\})
    \end{split}
\end{equation*}
Since $M_1>M_2$ we have $u\geq v$. We will now prove the result in cases based on the values of $u$ and $v$ and whether $u>v$ or $u=v$.\\

First we prove the lemma in the case that $u>v$ and $u = l$. Now we have $l\notin \{a,b,c,d,n\}$ so that $x_l\mid M_1/\gcd(M_1,M_2)$. If $(a,l) \neq (1,n-1)$, then $M_3:=\frac{x_ax_l}{x_ax_b}M_2$ works. So we may assume that $a=1$ and $l=n-1$. If $b\neq n-2$ then $M_3:= \frac{x_bx_{n-1}}{x_ax_b}M_2$ works, so we can assume further that $b = n-2$. If $i\neq 1,n-3$, then $M_3:= \frac{x_ix_b}{x_ax_b}M_2$ works, so we can assume that $i\in\{1,n-3\}$. Since $l=n-1$ and $\{i,l\}\in E(\mathcal{A}_{n-1})$ we have $i\neq 1$. So $i=n-3$. Note that $c,d\leq n-3$ since $\{c,n\},\{d,n\}\in E(F)$. If $c,d\neq n-3$ then $c,d<n-3$ so that $M_3:= \frac{x_{n-3}x_n}{x_{\min\{c,d\}}x_n}M_2$ works. Thus one of $c$ or $d$ equals $n-3$. But also $d\geq c$, so $d=n-3$. Now we make a few additional reductions. If $x_{n-3}^2\mid M_1$ then one of $j$ or $r$ equals $n-3$, contradicting (\ref{eq4}). Consequently we have $r,j<n-3$, so that $M_3:= \frac{x_rx_{n-2}}{x_1x_{n-2}}M_2$ works.\\

Now we will prove the result when $u>v$ and $u=r$. We have $r>a,b,c,d$ and $r\notin\{a,b,c,d,n\}$ so that $x_r\mid M_1/\gcd(M_1,M_2)$. Since $b<r\leq n-3$ we have $\{a,r\}\in E(G)$ by Remark \ref{rmk1.7}. Now $M_3:= \frac{x_ax_r}{x_ax_b}M_2$ works by Lemma \ref{lem1.8}(iii).\\

Now we may assume that $u = v$. By (\ref{eq4}) we have $(u,v)\neq (r,d)$. Thus $(u,v)\in \{(l,b),(l,d),(r,b)\}$. We proceed by individually treating these three cases.\\

Suppose that $l = u = v = b$. We have
\begin{equation*}
    \begin{split}
        M_1&= x_ix_lx_jx_nx_rx_n\\
        M_2&= x_ax_lx_cx_nx_dx_n.
    \end{split}
\end{equation*}
Let $p = \max\{i,j,r\}$. $p = \max\{i,r\}$ since $r\geq j$. If $p\notin \{a,b,c\}$ then $p = i$, so that $M_3:= \frac{x_ix_l}{x_ax_l}M_2$ works by Lemma \ref{lem1.8}(iii), or $p = r$ and $M_3:= \frac{x_rx_n}{x_dx_n}M_2$ works. So we can assume that $p\in\{a,b,c\}$. If $p = i$, then $x_r\mid M_1/\gcd(M_1,M_2)$. Consequently $M_3:= \frac{x_rx_n}{x_dx_n}M_2$ works. So we can assume that $p = r$. By (\ref{eq4}) we must have $p = a$. Now
\begin{equation*}
    \begin{split}
        M_1&= x_ix_lx_jx_nx_rx_n\\
        M_2&= x_rx_lx_cx_nx_dx_n.
    \end{split}
\end{equation*}
We have $z:=\max\{i,j\}>\max\{c,d\}$ by (\ref{eq5}). Then $M_3:=\frac{x_zx_n}{x_dx_n}M_2$ works by Lemma \ref{lem1.8}(iii).\\

Now we prove the lemma when $l = u = v = d$. Again let let $p = \max\{i,j,r\}$. If $p\notin\{a,b,c\}$ then $M_3:= \frac{x_px_n}{x_cx_n}M_2$ works by Lemma \ref{lem1.8}(iii). So we may assume that $p\in\{a,b,c\}$. But in this case $M_3:=\frac{x_zx_n}{x_dx_n}M_2$ works by Lemma \ref{lem1.8}(iii), via a similar argument to the $l = u = v = b$ case.

Now it remains to prove the lemma in the case that $r = u = v=b$. Let $p = \max\{l,j,i\}$ ($=\max\{l,j\}$). As before, if $p\notin\{a,b,c\}$ then $M_3:= \frac{x_px_n}{x_cx_n}M_2$ works. Suppose that $p\in \{a,b,c\}$ and let
\begin{equation*}
    z = \begin{cases}
        l \text{ }&\text{if }p = j\\
        j &\text{if } p =l
    \end{cases}
\end{equation*}
Then we have $z >\max\{a,c\}$ so that $M_3:= \frac{x_zx_n}{x_cx_n}M_2$ works by Lemma \ref{lem1.8}(iii) (we have $\{z,n\}\in E(F)$ since $z\leq n-3$). This completes the proof of Lemma \ref{lem2}.

\end{proof}

Now we prove Lemma \ref{lem3}, thereby completing the proof of Theorem \ref{cor1}.

    Let $\mathcal{O}$ be the set of minimal (monomial) generators of $I_G^2I_F$ equipped with the lexicographical order. As before, for $M,N\in \mathcal{O}$ we shall write $N>M$ to express the statement that $N>_{\lex}M$. We begin by establishing notation which will be useful throughout the proof. Let $M_1,M_2\in \mathcal{O}$ such that $M_1>M_2$. Write
    \begin{equation*}
        \begin{split}
            M_1&= x_ix_lx_jx_rx_sx_n\\
            M_2&= x_ax_bx_cx_dx_tx_n
        \end{split}
    \end{equation*}
    where $\{i,l\},\{j,r\},\{a,b\},\{c,d\}\in E(G)$, $i<l$, $j<r$, $a<b$, $c<d$, $l\leq r$, $b\leq d$, $\{s,n\}, \{t,n\}\in E(F)$. Then $1\leq i,j,a,c,s,t\leq n-3$ and $3\leq l,r,b,d\leq n-1$. We make two reductions (equation (\ref{eq6}) and equation (\ref{eq7})) before proceeding with the main force of the proof. By Lemma \ref{lem1} we can assume that
    \begin{equation}\label{eq6}
        \deg(\gcd(x_ix_lx_jx_rx_s,x_ax_bx_cx_dx_t))\leq 3.
    \end{equation}
    By a similar proof to that of (\ref{eq4}) in the proof of Lemma \ref{lem2}, we obtain the following statement
    \begin{equation}\label{eq7}
        \begin{split}
            &\text{Lemma \ref{lem3} holds in the case that}\\
            &(i,l)=(a,b)\text{, }(i,l)=(c,d)\text{, }
            (j,r)=(a,b)\text{, or }(j,r)=(c,d).
        \end{split}
    \end{equation}
Now we divide the proof into eight major cases. Let
\begin{equation*}
    \begin{split}
        u &= \max\{i,l,j,r,s\} = \max\{r,s\}\\
        v &= \max\{a,b,c,d,t\} = \max\{d,t\}.
    \end{split}
\end{equation*}
Since $M_1>M_2$ we have $u\geq v$. We obtain the following eight cases.
\begin{enumerate}
    \item[(I)]
    $r=u>v=d$
    \item[(II)]
    $r=u>v=t$
    \item[(III)]
    $s=u>v=d$
    \item[(IV)]
    $s=u>v=t$
     \item[(V)]
    $r=u=v=d$
    \item[(VI)]
    $r=u=v=t$
    \item[(VII)]
    $s=u=v=d$
    \item[(VIII)]
    $s=u=v=t$
\end{enumerate}
Now we introduce a bit more notation, followed by a remark. Let $p_1\leq p_2\leq p_3\leq p_4\leq p_5=u$ be the (possibly not distinct) elements $i,l,j,r,s$ arranged in ascending order and let $p = p_4$. Let $q_1\leq q_2\leq q_3\leq q_4\leq q_5 = v$ be the elements $a,b,c,d,t$ arranged in ascending order and let $q = q_4$. Let $y = p_3$ and let $z = p_2$. Note that $z\leq y\leq p\leq u$.

The following remark is a consequence of the fact that $M_1>M_2$.
\begin{remark}\label{rmk-p}
    We have the following statements.
    \begin{enumerate}
        \item[(1)] $u=v$ or $u>a,b,c,d,t$.
        \item[(2)] $p\geq q$.
        \item[(3)] $p= q$ or $p> q_1,q_2,q_3,q_4$. 
    \end{enumerate}
\end{remark}

Next we proceed by individually proving each of the eight cases in the order stated above.

\begin{proof}[proof of the $r=u>v=d$ case]

$r>a,b,c,d,t$ by Remark \ref{rmk-p}(1). Consequently, if $(c,r)\neq (1,n-1)$, then $M_3:= \frac{x_cx_r}{x_cx_d}M_2$ works by Lemma \ref{lem1.8}(iii). So we can assume that $c=1$ and $r=n-1$. Similarly we can assume that $a=1$. If $b$ and $d$ do not both equal $n-2$, then $M_3:= \frac{x_{\min\{b,d\}}x_r}{x_1x_{\min\{b,d\}}}M_2$ works by Lemma \ref{lem1.8}(iii). So we can assume that $b=d=n-2$. Thus
\begin{equation*}
    M_3:= \frac{x_{n-1}}{x_{n-2}}M_2
    = x_1x_{n}\cdot x_1x_{n-2}\cdot x_tx_{n-1}
\end{equation*}
works by Lemma \ref{lem1.8}(iii).

\end{proof}

\begin{proof}[proof of the $r=u>v=t$ case]
We have $r>a,b,c,d,t$. Similar to the proof of the $r=u>v=d$ case, we can assume that $r=n-1$, $a=c=1$, and $b=d=n-2$. Thus $M_3:= \frac{x_{n-1}}{x_{n-2}}M_2$ works (by Lemma \ref{lem1.8}(iii)).
\end{proof}

\begin{proof}[proof of the $s=u>v=d$ and $s=u>v=t$ cases]
We have $s>a,b,c,d,t$ by Remark \ref{rmk-p}(1). Then $M_3:= \frac{x_sx_n}{x_tx_n}M_2$ works by Lemma \ref{lem1.8}(iii).
\end{proof}

Before proving the remaining four ($u=v$) cases of Lemma \ref{lem3} we make some comments about the general strategy. Notation \ref{not1} and Lemma \ref{lem1.8} will be used repeatedly throughout the proof.
\begin{comments}
    In the paragraph preceding Remark \ref{rmk-p} we have taken $p$ to be the second largest (accounting for repetition) of the elements $i,l,j,r,s$ and we have taken $q$ to be the second largest of the elements $a,b,c,d,t$. In each of the four $u=v$ cases ((V),\ldots ,(VIII)) we shall first prove the lemma when $p\notin \{q_1,q_2,q_3,q_4=q\}$, or equivalently when $p>q$. In this case we shall find a suitable projection $M_3$ of $M_1$ onto $M_2$. In the case that $p = q$, more effort is required. We then have to compare the third and fourth largest of $i,l,j,r,s$, and those of $a,b,c,d,t$, and we find and $M_3$ that works by application of Lemma \ref{lem1.8}(iii). 
    
    Often we will write our \emph{candidate} in the style $M_3 = (x_{i_1}x_{i_2})(x_{i_3}x_{i_4})(x_{i_5}x_{i_6})$ to emphasize when $x_{i_1}x_{i_2},x_{i_3}x_{i_4}\in I_G$, and $x_{i_5}x_{i_6}\in I_F$ to clarify how $M_3\in I_G^2I_F$.
\end{comments}

\begin{proof}[proof of the $r=u=v=d$ case]
    By construction we have $p = \max\{i,j,s\}$ and $q = \max\{b,c,t\}$. Suppose that $p\notin \{b,c,t\}$. If $(a,l)\neq (1,n-1)$ or $p\neq l$, then by Lemma \ref{lem1.8}(iii) and Remark \ref{rmk1.7} we have that
    \begin{equation*}
        M_3
        :=
        \begin{cases}
            x_sx_nM_2/x_tx_n &\text{ if }p=s\\
            x_jx_rM_2/x_cx_r &\text{ if }p=j\\
            x_ax_lM_2/x_ax_b &\text{ if }p=l
            \text{ and }(a,l)\neq (1,n-1)
        \end{cases}
    \end{equation*}
works. Now suppose further that $(a,l)=(1,n-1)$. If $b\neq n-2$ then $M_3:= \frac{x_{n-1}x_b}{x_ax_b}M_2$ works. So we may assume that $b=n-2$. By (\ref{eq7}) we can assume that $(\dagger)$: $j\neq c$. Since $\{j,n-1\}\in E(\mathcal{A}_{n-1})$ we have $2\leq j\leq n-3$. Thus, if $j\notin\{c,t\}$, then $M_3:= \frac{x_jx_{n-2}}{x_1x_{n-2}}M_2$ works. So we can assume that $j\in \{c,t\}$. Now by $(\dagger)$ we have that $j=t$. But now
\begin{equation*}
    \begin{split}
        M_1 &= x_ix_{n-1}x_jx_{n-1}x_sx_n\\
        M_2 &= x_1x_{n-2}x_jx_{n-2}x_cx_n
    \end{split}
\end{equation*}
and we are done by (\ref{eq7}). Now we have reduced to the case that $p\in\{b,c,t\}$, so assume this.\\

Suppose that $p = c$. By (\ref{eq7}) we can assume that $p\neq j$, so that $p\in \{l,s\}$ and $j<p$. Suppose further that $p=c=l$. By construction, $y = \max\{i,j,s\}$ and
\begin{equation*}
    z=\begin{cases}
        \max\{i,j\} &\text{ if }y = s\\
        \max\{i,s\} &\text{ if }y=j\\
        \max\{j,s\} &\text{ if }y=i
    \end{cases}
\end{equation*}
We have $z\leq y\leq p=c\leq n-3$. If $y\in \{b,c\}$, then $z>q_1,q_2$, due to (\ref{eq6}) and the fact that $M_1>M_2$. Now we have the following remarks.
\begin{enumerate}
    \item[(i)] $z>a$.
    \item[(ii)] If $y = b$, then $z>t$.
    \item[(iii)] If $y = t$, then $z>b$.
\end{enumerate}
Note that by (\ref{eq7}) the lemma holds in the following cases.
\begin{enumerate}
    \item[(1)] $i=y=b$, in which case $M_2 = x_yx_px_ax_rx_tx_n$
    \item[(2)] $j=y=b$, in which case $M_2 = x_ax_px_yx_rx_tx_n$
    \item[(3)] $j=y=t$ or $s=y=t$, in which case $M_2 = x_ax_bx_yx_rx_px_n$. 
\end{enumerate}
So we may assume that none of (1), (2), or (3) hold. Now by Lemma \ref{lem1.8}(iii) and Remark \ref{rmk1.7}, the monomial
\begin{equation*}
    M_3:=\begin{cases}
        \frac{x_yx_n}{x_tx_n}M_2 &\text{ if }y\in\{b,c\}\\
        \frac{x_i}{x_a}M_2 = (x_ix_p)(x_tx_r)(x_yx_n)
        &\text{ if } s=y=b\text{ and }z=i\\
        \frac{x_j}{x_t}M_2
        = (x_ax_p)(x_jx_r)(x_yx_n) &\text{ if }
        s=y=b\text{ and }z = j\\
        \frac{x_ax_z}{x_ax_b}M_2&\text{ if }
        i=y=t
    \end{cases}
\end{equation*}
works. This completes our proof in the case that $c=p=l$.\\

If $c=p=s$, then $M_1 = x_ix_l\cdot x_px_r \cdot x_jx_n$ and $M_2 = x_ax_bx_px_rx_tx_n$ and we are done by (\ref{eq7}). Suppose that $p = b= l$. We have, $y = \max\{i,j,s\}$,
\begin{equation*}
    z=\begin{cases}
        \max\{i,j\} &\text{ if }y = s\\
        \max\{i,s\} &\text{ if }y=j\\
        \max\{j,s\} &\text{ if }y=i
    \end{cases}
\end{equation*}
and $z\leq y\leq n-3$, since $i,j,s\leq n-3$. We have the following remarks.
\begin{enumerate}
    \item[(i)] $y\in \{a,c,t\}$ or $y>a,c,t$.
    \item[(ii)] If $y = a$, then $z>c,t$.
    \item[(iii)] If $y = c$, then $z>a,t$.
    \item[(iv)] If $y = t$, then $z>a,c$. 
\end{enumerate}
By (\ref{eq7}), we have $i\neq a$ and $j\neq c$, and we are done if $i=y=t$ or $j=y=t$. So we can assume that neither of these cases hold. Using the above statements (i), (ii), (iii), (iv), Lemma \ref{lem1.8}(iii), and Remark \ref{rmk1.7}, we see that the following monomial
\begin{equation*}
    M_3:=\begin{cases}
        x_ix_pM_2/x_ax_p &\text{ if }i=y\notin\{a,c,t\} \text{ or [$s=y=t$ and $z=i$]}\\
        &\text{ or [$s=y=c$ and $z=i$]}\\
        x_jx_rM_2/x_cx_r &\text{ if }j=y\notin\{a,c,t\} \text{ or [$s=y=t$ and $z=j$]}\\
        &\text{ or [$s=y=a$ and $z=j$]}\\
        x_sx_nM_2/ x_tx_n &\text{ if }s=y\notin\{a,c,t\} \text{ or [$j=y=a$ and $z=s$]}\\
        x_zx_nM_2/x_tx_n &\text{ if } [s=y=c \text{ and }z=j]
        \text{ or [$s=y=a$ and $z=i$]}\\
        x_zx_rM_2/x_cx_r &\text{ if } j=y=a \text{ and }z=i\\
        x_zx_nM_2/x_tx_n &\text{ if } y=i=c
    \end{cases}
\end{equation*}
works.\\

Suppose that $y=p=b$. Since $i<l$, we have $y =\max\{i,l,s\}= \max\{l,s\}$ and
\begin{equation*}
    z=\begin{cases}
        \max\{i,l\}=l &\text{ if }y=s\\
        \max\{i,s\} &\text{ if }y=l
    \end{cases}
\end{equation*}
We have either $z\in \{i,s\}$ or $z \leq y = s$, while $s\leq n-3$. Thus $z\leq n-3$. By (\ref{eq7}) the lemma holds if $s=y=t$. So we can assume that $s$ or $t$ does not equal $y$. Now by Lemma \ref{lem1.8}(iii) and Remark \ref{rmk1.7}, the monomial

\begin{equation*}
    M_3:=
    \begin{cases}
        x_yx_nM_2/x_tx_n &\text{ if }y\notin\{a,c,t\}\\
        x_zx_nM_2/x_tx_n &\text{ if }y\in\{a,c\}\\
        x_zx_rM_2/x_cx_r &\text{ if }l=y=t\text{ and }z=i\\
        x_zx_nM_2/x_cx_n &\text{ if }l=y=t \text{ and }z=s\\
    \end{cases}
\end{equation*}
works (e.g. to see the $y\notin \{a,c,t\}$ case, we have $y\leq p\leq u$ implies that $\ord_{x_y}M_3 \geq 3$ and $\ord_{x_y}(M_2)\leq 2$ so that $x_y\mid M_1/\gcd(M_1,M_2)$).\\

Suppose that $s=p=b$. We have $y = \max\{l,j\}$ and
\begin{equation*}
    z = \begin{cases}
        l &\text{ if }y=j\\
        \max\{i,j\} &\text{ if }y=l
    \end{cases}
\end{equation*}
We have $z \leq n-3$. Additionally, $y\in \{a,c,t\}$ or $y>a,c,t$. If $y = a$ then $z>c,t$, if $y = c$ then $z>a,t$, and if $y=t$ then $z>a,c$. Since $y\leq p=s\leq n-3$ we have that $\{y,n\}\in E(F)$. Similarly $\{z,n\}\in E(F)$. By (\ref{eq7}) we can assume that $j$ and $t$ do not both equal $y$. Then by Lemma \ref{lem1.8}(iii) and Remark \ref{rmk1.7} we have that
\begin{equation*}
    M_3:=
    \begin{cases}
        x_yx_nM_2/x_tx_n &\text{ if }y\notin\{a,c,t\}\\
        x_zx_nM_2/x_tx_n &\text{ if }y\in\{a,c\}\\
        x_zx_pM_2/x_ax_p &\text{ if }l=y=t\text{ and }z=i\\
        x_zx_rM_2/x_cx_r &\text{ if } l=y=t\text{ and }z=j
    \end{cases}
\end{equation*}
works.\\

Suppose that $p = t= l$. Then $y = \max\{i,j,s\}$,
\begin{equation*}
    z =
    \begin{cases}
        \max\{i,j\}&\text{ if }y=s\\
        \max\{i,s\}&\text{ if }y=j\\
        \max\{j,s\}&\text{ if }y=i
    \end{cases}
\end{equation*}
and $z\leq y\leq p\leq n-3$ (since $t\leq n-3$). We also see that $z>a$, $y\in \{b,c\}$ or $y>b,c$, if $y=b$ then $z>c$, and if $y = c$ then $z>b$. Then by Lemma \ref{lem1.8}(iii) and Remark \ref{rmk1.7} we find that the monomial

\begin{equation*}
    M_3:=
    \begin{cases}
        \frac{x_y}{x_b}M_2 = (x_ax_p)(x_cx_r)(x_bx_n)
        &\text{ if }y\notin\{b,c\}\\
        \frac{x_i}{x_a}M_2 = (x_ix_p)(x_cx_r)(x_bx_n)
        &\text{ if }i=y=b\\
        \frac{x_j}{x_a}M_2 = (x_ax_p)(x_jx_r)(x_bx_n)
        &\text{ if }j=y=b\\
        \frac{x_z}{x_a}M_2 = (x_ix_p)(x_cx_r)(x_yx_n)
        &\text{ if } s=y=b \text{ and }z=i\\
        \frac{x_z}{x_c}M_2 = (x_ax_p)(x_jx_r)(x_yx_n)
        &\text{ if }s=y=b \text{ and }z=j\\
        \frac{x_zx_n}{x_bx_n}M_2
        &\text{ if }y= c
    \end{cases}
\end{equation*}
works.\\

If $p=j=t$, then
\begin{equation*}
    \begin{split}
        M_1&= x_ix_lx_px_rx_sx_n\\
        M_2&= x_ax_bx_px_rx_cx_n
    \end{split}
\end{equation*}
and we are done by (\ref{eq7}). Suppose that $p=t=s$. Then $z\leq y\leq p\leq n-3$, $y = \max\{j,l\}$, and $z = \begin{cases}l &\text{ if }y=j\\ \max\{i,j\}&\text{ if }y=l \end{cases}$. In addition, $z>a$, $y\in \{b,c\}$ or $y>b,c$, if $y = b$ then $z>c$, and if $y = c$ then $z>b$. We can assume that $j$ and $b$ do not both equal $y$ by (\ref{eq7}). These facts combined with Lemma \ref{lem1.8}(iii) and Remark \ref{rmk1.7} show that
\begin{equation*}
    M_3:=
    \begin{cases}
        x_yx_nM_2/x_bx_n &\text{ if }y\notin\{b,c\}\\
        x_zx_nM_2/x_bx_n &\text{ if }y=c\\
        \frac{x_j}{x_c}M_2 = (x_ax_p)(x_jx_r)(x_yx_n)
        &\text{ if } l=y=b\text{ and }z=j\\
        \frac{x_i}{x_a}M_2 = (x_ix_p)(x_cx_r)(x_yx_n)
        &\text{ if }l=y=b \text{ and }z=i
    \end{cases}
\end{equation*}
works. This completes the proof of Lemma \ref{lem3} in the case that $r=u=v=d$.
\end{proof}

\begin{proof}[proof of the $r=u=v=t$ case] Recall that $z$, $y$, $p$, and $u$ are resp. the fourth largest, third largest, second largest, and largest indices $i$ such that $x_i\mid M_1$ (besides $n$), and $v$ is the largest index $i$ (besides $n$) such that $x_i\mid M_2$.

    We have $p = \max\{l,j,s\}\geq \max\{a,b,c,d\} = d$ by construction. Suppose that $p\neq d$. Then $p>a,b,c,d$ by Remark \ref{rmk-p}(3). Consequently, if $p = s$, then $M_3:= \frac{x_cx_p}{x_cx_d}M_2$ works. If $p =j$ or [$p=l$ and $(a,l)\neq(1,n-1)$], then $p\leq n-3$ so that $M_3:= \frac{x_ax_p}{x_ax_b}M_2$ works by Lemma \ref{lem1.8}(iii) and Remark \ref{rmk1.7}. Suppose that $p=l$ and $(a,l)=(1,n-1)$. If $b$ does not equal $n-2$, then $M_3:= \frac{x_{b}x_{n-1}}{x_ax_b}M_2$ works. If $d$ does not equal $n-2$ then $M_3:= \frac{x_dx_{n-1}}{x_cx_d}M_2$ works. If $c\neq 1$, then $M_3:= \frac{x_cx_{n-1}}{x_cx_d}M_2$ works. If $b=d=n-2$ and $a=c=1$, then $x_{n-1}\mid M_1/\gcd(M_1,M_2)$ so that
    \begin{equation*}
        M_3:=\frac{x_{n-1}}{x_{n-2}}M_2
        = x_1x_{n-2}x_tx_{n-1}x_1x_n
    \end{equation*}works.\\

    Now we may assume that $p=d$. Suppose that $p = s$. Then $y = \max\{l,j\}$ and
    \begin{equation*}
        z = 
        \begin{cases}
            \max\{i,j\} &\text{ if }y=l\\
            l &\text{ if } y=j
        \end{cases}
    \end{equation*}

 Then we have $z\leq n-3$, $z>a$, if $y=b$ then $z>c$, and if $y = c$ then $z>b$. By (\ref{eq7}) and the fact that $y\in \{l,j\}$, we can assume that none of the following statements hold.
 \begin{enumerate}
     \item[(1)] $y=b=j$.
     \item[(2)] $y = c$.
 \end{enumerate}
 Consequently, the monomial
 \begin{equation*}
     M_3:=
     \begin{cases}
     x_ax_yM_2/x_ax_b &\text{ if }y\notin\{b,c\}\\
         x_ix_bM_2/x_ax_b &\text{ if }y=b=l \text{ and }z=i\\
         \frac{x_z}{x_c}M_2 &\text{ if } y=b=l\text{ and }z=j
     \end{cases}
 \end{equation*}
 works (By Lemma \ref{lem1.8}(iii) and Remark \ref{rmk1.7}).\\

Suppose that $p = l$. Then $z\leq y\leq p\leq n-3$, $y = \max\{i,j,s\}$, and
\begin{equation*}
    z =
    \begin{cases}
        \max\{i,j\} &\text{ if }y=s\\
        \max\{i,s\} &\text{ if }y=j\\
        \max\{j,s\} &\text{ if }y=i
    \end{cases}
\end{equation*}
Moreover, $z>a$, $y\in\{b,c\}$ or $y>b,c$, if $y = b$ then $z>c$, and if $y = c$ then $z>b$. By (\ref{eq7}) we can assume that the statements $y=b=j$, $c=y=i$, and $y=c=j$ do not hold. Now
\begin{equation*}
    M_3:=
    \begin{cases}
    x_ax_yM_2/x_ax_b &\text{ if }y\notin\{b,c\}\\
        \frac{x_z}{x_c}M_2 = (x_yx_p)(x_ax_r)(x_zx_n) &\text{ if } y=b=i\\
        \frac{x_z}{x_a}M_2 = (x_zx_r)(x_cx_p)(x_yx_n) &\text{ if } y=b=s\\
        \frac{x_z}{x_a}M_2 = (x_zx_r)(x_yx_p)(x_bx_n) &\text{ if }y=c=s
    \end{cases}
\end{equation*}
works.\\

Suppose that $p=j$. Then $z\leq y\leq p\leq n-3$, $y = \max\{l,s\}$, and
\begin{equation*}
    z=
    \begin{cases}
        \max\{i,s\} &\text{ if }y=l\\
        l &\text{ if }y=s
    \end{cases}
\end{equation*}
Furthermore, $z>a$, $y\in\{b,c\}$ or $y>b,c$, if $y = b$ then $z>c$, and if $y = c$ then $z>b$. By (\ref{eq7}) we can assume that $y\notin\{b,c\}$ (we see that (\ref{eq7}) implies the lemma in each case $l=y=b$, $s=y=b$, $l=y=c$, $s=y=c$). Consequently, $M_3: = \frac{x_ax_y}{x_ax_b}M_2$ works by Remark \ref{rmk1.7} and Lemma \ref{lem1.8}(iii). This completes the proof of the lemma in the case that $r=u=v=t$.
 \end{proof}

\begin{proof}[proof of the $s=u=v=d$ case]
We see that $r = p$ in this case (by construction of $r$). In addition, $z\leq y\leq p=r\leq s\leq n-3$. Thus $\{r,n\}\in E(F)$. If $r\notin \{b,c,t\}$ then $M_3:= \frac{x_rx_n}{x_tx_n}M_3$ works by Lemma \ref{lem1.8}(iii).\\

So we can assume that $r\in\{b,c,t\}$. Note that $y = \max\{j,l\}$ and
\begin{equation*}
    z =
    \begin{cases}
        l &\text{ if }y=j\\
        \max\{i,j\} &\text{ if }y =l
    \end{cases}
\end{equation*}

Suppose that $r = b$. Then $y\in\{a,c,t\}$ or $y>a,c,t$, if $y=a$ then $z>c,t$, if $y=c$ then $z>a,t$, and if $y=t$ then $z>a,c$. By (\ref{eq7}) we can assume that $t$ and $j$ do not both equal $y$. Now we find that the monomial

\begin{equation*}
    M_3:=
    \begin{cases}
        x_yx_nM_2/x_tx_n &\text{ if }y\notin\{a,c,t\}\\
        x_zx_nM_2/x_tx_n &\text{ if }y\in\{a,c\}\\
        x_jx_rM_2/x_ax_r &\text{ if }l=j=t\text{ and }z=j\\
        \frac{x_i}{x_a}M_2
        =(x_ix_y)(x_cx_s)(x_rx_n) &\text{ if } l=y=t \text{ and }z=i
    \end{cases}
\end{equation*}
works.\\

If $r = c$, then
\begin{equation*}
    \begin{split}
        M_1 &= x_ix_l\cdot x_rx_s\cdot x_jx_n\\
        M_2 &= x_ax_b\cdot x_rx_s\cdot x_bx_n
    \end{split}
\end{equation*}
so that we are done by (\ref{eq7}). We may now assume that $r=t$. Now we have $z>a$, $y\in \{b,c\}$ or $y>b,c$, if $y=b$ then $z>c$, and if $y=c$ then $z>b$. By (\ref{eq7}) we can assume that none of $j=y=b$, $l=y=c$, or $j=y=c$ hold. Then we claim that
\begin{equation*}
    M_3:=
    \begin{cases}
        x_ax_yM_2/x_ax_b &\text{ if }y\notin\{b,c\}\\
        x_ix_yM_2/x_ax_y &\text{ if }l=y=b \text{ and }z=i\\
        x_jx_sM_2/x_cx_s &\text{ if }l=y=b\text{ and }z=j
    \end{cases}
\end{equation*}
works, we give a brief proof that $(\dagger):$ $M_3:= x_ax_yM_2/x_ax_b$ works if $y\notin\{b,c\}$ now.
\begin{proof}[proof of $(\dagger)$]
    Suppose that $y\neq \{b,c\}$. Recall that we have reduced to the case that $y=\max\{j,l\}$, $r=t$, and now we also have $y>a,b,c$ ($y>b,c$ by the preceding paragraph, while $b>a$). We will show that $\ord_{x_y}M_1>\ord_{x_y}M_2$, which clearly holds if $x_y\nmid M_2$, since $x_y\mid M_1$. So we may assume that $x_y\mid M_2$, i.e. $y\in\{a,b,c,d,t\}$, and thus $y\in \{d,t\} = \{s,t\}$.

    First we address the case when $y = s$. So $y = s= d$. But $y\leq p\leq u$, (by construction of $y$, $p$, and $u$) so we must have $\ord_{x_y} (M_1)\geq 3$. But $y>a,b,c$ and $y\neq n$ so the largest $\ord_{x_y}(M_2)$ could be is 2.

    Now we can assume that $y\neq s$. Since $y\in\{t,s\}$ and $y\neq s$, we have $y = t< s$ (since $s = v$). Now $y = r = t$ while we have $y = \max\{j,l\}$, so that $\ord_{x_y}(M_1)\geq 2$. On the other hand, $y\neq s = d$ and $y<n$, and $y>a,b,c$ so $ord_{x_y}(M_2) = 1$.
\end{proof}

This completes the proof of the $s=u=v=d$ case.
\end{proof}

\begin{proof}[proof of the $s=u=v=t$ case]
We have $r = p\geq q = d$. Suppose that $r>d$. We have $r\leq u = s\leq n-3$, and in particular $r\neq n-1$. Then $M_3:=\frac{x_cx_r}{x_cx_d}M_2$ works by Remark \ref{rmk1.7} (and Lemma \ref{lem1.8}(iii)). Hence we can assume that $r= d$. Now $y = \max\{l,j\}\geq \max\{a,b,c\}$. Suppose that $y\notin\{a,b,c\}$. If $y = j$, then $M_3:= \frac{x_jx_r}{x_cx_r}M_2$ works. If $y = l$, then $M_3:= \frac{x_ax_l}{x_ax_b}M_2$ works. So we can assume that $y\in\{a,b,c\}$. Since $y = \max\{a,b,c\}$ we must have $y\neq a$, and so $y\in\{b,c\}$. To review, we have reduced to the case that $r=p=d$ and $y\in\{b,c\}$.\\

Suppose that $y = c$. Then $z>b$, so that $M_3: = \frac{x_ax_z}{x_ax_b}M_2$ works. Thus we may assume that $y = b$. Suppose that $y=j$. In this case we have $l = z>c$. Since $\{j,r\}\in E(G)$, $l\leq j$, and $r\leq s\leq n-3$, we have that $\{l,r\}\in E(G)$ by Remark \ref{rmk1.7}. This combined with the fact that $l>c$ yields that $M_3:= \frac{x_lx_r}{x_cx_r}M_2$ works. Now we have reduced to the case that $b=y = l$. Then
\begin{equation*}
    M_3:=
    \begin{cases}
        x_ix_yM_2/x_ax_y &\text{ if }z=i\\
        x_jx_rM_2/x_cx_r &\text{ if }z=j
    \end{cases}
\end{equation*}
works. This completes the proof of the lemma in the case that $s=u=v=t$.
\end{proof}
This completes the proof of Lemma \ref{lem3}, and hence the proof of Theorem \ref{cor1} is complete.

\section*{Acknowledgments}

The author would like to thank the Center for Mathematical Sciences and Applications at Harvard and the Hebrew University of Jerusalem for their support during this project. The author would also like to thank Eran Nevo for his helpful discussion on linear quotient orderings. The author would like to thank Adam Van Tuyl for testing some examples of linear quotient orderings in Macaulay2.

\bibliographystyle{amsplain}

\bibliography{biblio}

\end{document}